\documentclass[11pt,oneside,english]{amsart}
\usepackage{mathptmx}
\usepackage[T1]{fontenc}
\usepackage[latin9]{inputenc}
\usepackage{babel}

\usepackage{textcomp}
\usepackage{amsthm}
\usepackage{amssymb}
\usepackage{esint}
\usepackage[unicode=true, 
 bookmarks=true,bookmarksnumbered=false,bookmarksopen=false,
 breaklinks=false,pdfborder={0 0 1},backref=false,colorlinks=false]
 {hyperref}
\hypersetup{pdftitle={Equivalence of Euler-Lagrange and Noether equations },
 pdfauthor={Apostol Constantino Faliagas}}

\makeatletter

\providecommand{\tabularnewline}{\\}

  \theoremstyle{definition}
  \newtheorem{defn}{Definition}
  \theoremstyle{plain}
  \newtheorem{prop}{Proposition}
  \theoremstyle{remark}
  \newtheorem{rem}{Remark}
\theoremstyle{plain}
\newtheorem{thm}{Theorem}
  \theoremstyle{plain}
  \newtheorem{cor}{Corollary}
 \theoremstyle{definition}
  \newtheorem{example}{Example}

\theoremstyle{definition}
\newtheorem{cexample}{Counterexample}

\makeatother

\begin{document}

\title{On the Equivalence of Euler-Lagrange and Noether Equations}
\begin{abstract}
We prove that on the condition of non-trivial solutions, the Euler-Lagrange
and Noether equations are equivalent for the variational problem of
nonlinear Poisson equation and a class of more general Lagrangians,
including position independent and of p-Laplacian type. As applications
we prove certain propositions concerning the nonlinear Poisson equation
and its generalisations, the equivalence of admissible and inner variations
and discuss the inverse problem of determining the Lagrangian from
conservation or symmetry laws.
\end{abstract}

\author{A. C. Faliagas}

\address{Department of Mathematics, University of Athens, Panepistemiopolis,
15784 Athens, Greece.}

\email{afaliaga@math.uoa.gr}

\thanks{I would like to express my thanks to N. Alikakos for reading this
paper and suggesting many useful improvements.}

\maketitle

\section{Introduction}

Alikakos in \cite{AL} presented a new method, according to which
Derrick-Pohozaev identities and monotonicity formulas can be derived
from energy-momentum tensors and reported some interesting applications
to the nonlinear Poisson equation. In subsequent work, Alikakos and
the author \cite{ALFAL,FAL} extended these ideas by deriving Derrick-Pohozaev
identities and monotonicity formulas in a more systematic fashion
and for a larger variety of Lagrangians.

As remarked by the author in \cite{FAL}, Derrick-Pohozaev identities
and monotonicity formulas, derived by the aforementioned method, seemed
to be more general than those derived by previous methods. More precisely,
the starting point for the derivation of Derrick-Pohozaev identity
by classical methods is the Euler-Lagrange equations. In the previously
mentioned method by Alikakos and the author, however, the Euler-Lagrange
equations can be replaced by Noether's equations \cite{NOE}, which
is a weaker hypothesis when $u$ is a classical $\mathrm{C}^{2}$
solution (\cite{GIA-1}, Chapter 3). A simple proof of this fact \cite{GIA-1}
is provided by Counterexample \ref{cexa:2} below.

The foregoing counterexample relies on trivial solutions of Noether's
equations. It therefore seemed meaningful to attempt the construction
of counterexamples involving non-trivial solutions. In my attempt,
I reached the surprising conclusion that, for a rather large class
of Lagrangians, only counterexamples involving trivial solutions are
possible. In all other cases the Euler-Lagrange and Noether's equations
are equivalent. The proof of this statement is the main result of
this paper. Theorem \ref{thm:NL-Poi-1} proves for the Lagrangian
of the nonlinear Poisson equation that every non-trivial classical
solution of Noether's system is necessarily a solution of the Euler-Lagrange
equation. Theorem \ref{thm:2} states a condition under which the
Euler-Lagrange and Noether's equations are equivalent. Theorems \ref{thm:Gen-1}
and \ref{thm:4} extend the basic idea to more general Lagrangians.
In Section \ref{sec:2} there is a review of inner variations, energy-momentum
tensors and other prerequisite material, serving mainly to introduce
notations. More details on these topics are found in \cite{GIA-1,FAL}.

Finally, there is yet another reason for considering the undertaking
of this research meaningful. According to Noether's theorem, given
any Lagrangian, there is a conservation law corresponding to each
continuous {}``symmetry transformation'', as are called in the physical
literature transformations leaving the Lagrangian and the equations
of motion invariant in form (\cite{FEY1}, �52-3). This permits observed
selection rules in nature to be directly transposed into symmetry
requirements on the Lagrangian and is used by physicists as a guide
for the introduction of interaction terms when developing new Lagrangians
in quantum field theory (\cite{BJDR}, �11.4, p. 17). But if we know
that a set of conservation laws is equivalent to the equations of
motion (i.e. the Euler-Lagrange equations), this should already determine
the Lagrangian or at least a class of equivalent Lagrangians. We demonstrate
this by an example in Section \ref{sec:Apps}.

In this paper, we are concerned with the purely technical mathematical-analytical
aspects of the subject, leaving the applications to physics for subsequent
work.

\section{\label{sec:2}Noether's Equations}

In this section we summarise background material which is necessary
for the understanding of the statement and proof of main results and
serves as a means for the introduction of notations to be used. The
standard reference is \cite{GIA-1}.

\subsection{General notation}

Throughout this paper $\Omega$ is a domain (open connected subset)
of $\mathbb{R}^{N}$, except when otherwise stated. The following
abbreviated notation \[
u_{,i}=\tfrac{\partial u}{\partial x_{i}}\]
is used for partial derivatives. Einstein's summation convention applies
everywhere, except when the contrary is explicitly mentioned.

Following standard notation, $\mathrm{C}^{r}(\Omega)$ is the set
of $r$ times continuously differentiable functions in $\Omega$ and
$\mathrm{C}^{r}(\overline{\Omega})$ the set of restrictions to $\Omega$
of $r$ times continuously differentiable functions in $\mathbb{R}^{N}$.
The set of $r$ times continuously differentiable, $\mathbb{R}^{M}$
(or $\mathbb{C}^{M}$) valued functions in $\Omega$ is denoted by
$\mathrm{C}^{r}(\Omega)^{M}$ and the corresponding set of restrictions
to $\Omega$ of $r$ times continuously differentiable functions in
$\mathbb{R}^{N}$, $\mathrm{C}^{r}(\overline{\Omega})^{M}$. $\mathcal{D}(\Omega)$
denotes the set of real (or complex) $\mathrm{C}^{\infty}$ functions
on $\Omega$ with compact support in $\Omega$ and $\mathcal{D}(\Omega)^{M}$
the set of $\mathbb{R}^{M}$ (or $\mathbb{C}^{M}$) valued $\mathrm{C}^{\infty}$
functions on $\Omega$ with compact support in $\Omega$.

\subsection{\label{sub:2-2}Variational functionals}

We will be considering (nonlinear) functionals $J:\mathrm{C}^{1}(\overline{\Omega})^{M}\to\mathbb{R}$,
$M\in\mathbb{N}$, of the form\medskip{}

\textsf{(VF)}\hfill{}$J(u):=\int_{\Omega}L(x,u(x),Du(x))dx,$\hfill{}\medskip{}

where $\Omega$ is a bounded domain of $\mathbb{R}^{N}$, $L(x,y,z)$
a Lagrangian,\[
L:\Omega\times\mathbb{R}^{M}\times\mathbb{R}^{N\cdot M}\to\mathbb{R},\]
$L\in\mathrm{C}^{1}(\overline{\Omega}\times\mathbb{R}^{M}\times\mathbb{R}^{N\cdot M})$
and $u\in\mathrm{C}^{1}(\overline{\Omega})^{M}$.

A function $u\in\mathrm{C}^{1}(\Omega)^{M}$ is a \emph{critical point}
of $J$ when\[
\delta J(u)v:=\left.\tfrac{d}{dt}J(u+tv)\right|_{t=0}=0\quad\forall v\in\mathcal{D}(\Omega)^{N}.\]
The derivative $\delta J(u)v$ is the \emph{variation of $J$ at $u$
in direction} $v$.

When $u\in\mathrm{C}^{2}(\Omega)^{M}$, an easy calculation shows\[
\delta J(u)v=\int_{\Omega}\delta L(u)\cdot vdx\]
where $\delta L(u)=(\delta L(u)_{i})_{i=1,\cdots,M}$ is the vector
field with components\begin{equation}
\delta L(u)_{i}=\left.\left(L_{y_{i}}-\tfrac{\partial}{\partial x_{j}}L_{z_{ij}}\right)\right|_{(x,u(x),Du(x))}.\label{eq:0}\end{equation}
We will refer to $\delta L$ as the \emph{Euler-Lagrange derivative}.
Every critical point $u\in\mathrm{C}^{2}(\Omega)^{M}$ of $J$ satisfies
the \emph{Euler-Lagrange equations}\[
\mbox{\ensuremath{\delta}}L(u)=0.\]

\subsection{\label{sub:2-3}Inner variations }

Inner variations are a special kind of variations. Let $\mathrm{I}=]-\delta,\delta[$,
$\delta>0$. Fixing a $u\in\mathrm{C}^{1}(\Omega)^{M}$ and a set
of diffeomorphisms $(\xi^{t})_{t\in\mathrm{I}}$, $\xi^{t}:\Omega\to\Omega$,
the following set of functions\[
\widetilde{u}(x,t):=u(\xi^{t}(x)),\quad t\in\mathrm{I}\]
define a set of variations of $u$ under certain conditions, to be
made precise in the following definition.
\begin{defn}
\textsf{A)} Let $h\in\mathcal{D}(\Omega)^{N}$, $\delta>0$ and $\mathrm{I}=]-\delta,\delta[$.
A set of diffeomorphisms $(\xi^{t})_{t\in\mathrm{I}}$ of $\Omega$
having the properties \textsf{(i)} - \textsf{(iii)} below and such
that the function $\xi:\Omega\times\mathrm{I}\to\Omega$ , $\xi(x,t)=\xi^{t}(x)$
is $\mathrm{C}^{\infty}$-differentiable, is called an \emph{inner
variation of }$\Omega$\emph{ in direction} $h$ or \emph{which is
defined by} $h$:

\textsf{(i)} $\xi^{0}=id_{\Omega}$, i.e. $\xi^{0}(x)=x$ in $\Omega$.

\textsf{(ii)} $D_{t}\xi(x,0)=h(x)\;\forall x\in\Omega$.

\textsf{(iii)} $\xi^{t}|\partial\Omega=id_{\Omega}$, i.e. $\xi^{t}(x)=x\;\forall x\in\partial\Omega$.

\textsf{B)} Let $J$ be a functional satisfying (\textsf{VF}), $u\in\mathrm{C}^{1}(\overline{\Omega})^{M}$,
$h\in\mathcal{D}(\Omega)^{N}$ and $(\xi_{h}^{t})_{t\in\mathrm{I}}$
the inner variation of $\Omega$ defined by $h$. The set of functions\[
u\circ\xi_{h}^{t},\quad t\in\mathrm{I}\]
is called the \emph{inner variation of} \emph{$u$ in direction }$h$.
The derivative\begin{equation}
\mathfrak{d}J(u)h:=\left.\tfrac{d}{dt}J(u\circ\xi_{h}^{t})\right|_{t=0},\label{eq:0-1}\end{equation}
is called the \emph{inner variation of the functional $J$ at $u$
in direction }$h$.\hfill{}$\square$
\end{defn}
For the calculation of inner variations the following proposition
is used.
\begin{prop}
\label{pro:invar}Let\textup{ }$J$ be a functional satisfying \textsf{(VF)}
and \textup{$u\in\mathrm{C}^{1}(\overline{\Omega})^{M}$}. The inner
variation of $J$ at $u$ is given by\begin{equation}
\mathfrak{d}J(u)h=\int_{\Omega}\left(u_{k,i}L_{z_{kj}}h_{i,j}-L\mathrm{div}h-L_{x_{i}}h_{i}\right)dx,\quad h\in\mathcal{D}(\Omega)^{N},\label{eq:10}\end{equation}
where $L$, $L_{x_{i}}:=\tfrac{\partial L}{\partial x_{i}}$, $L_{z_{kj}}:=\tfrac{\partial L}{\partial z_{kj}}$
are taken at the point $(x,u(x),Du(x))$, i.e. $L=L(x,u(x),Du(x))$,
$L_{x_{i}}=L_{x_{i}}(x,u(x),Du(x))$ etc.\end{prop}
\begin{proof}
The proof consists in considering the function $\varphi(t):=J(u\circ\xi^{t})$,
applying the change of integration variable $x=\eta^{t}(y)$, where
$\eta^{t}$ is the inverse function of $\xi^{t}$, differentiating
with respect to $t$ and interchanging differentiation with integration.
The details are lengthy and have been omitted as they are available
in \cite{FAL}.
\end{proof}

\subsection{\label{sub:2-4}Energy-momentum tensor }

On using the notation\begin{equation}
\mathfrak{d}L(u)h:=u_{k,i}L_{z_{kj}}(\cdot,u,Du)h_{i,j}-L(\cdot,u,Du)h_{i,i}-L_{x_{i}}(\cdot,u,Du)h_{i}\label{eq:12}\end{equation}
the expression for the inner variation of $J$ reduces to\begin{equation}
\mathfrak{d}J(u)h=\int_{\Omega}\mathfrak{d}L(u)hdx,\quad h\in\mathcal{D}(\Omega)^{N}.\label{eq:12c}\end{equation}

When $L_{x}=0$, formula (\ref{eq:12}) simplifies further to\[
\mathfrak{d}L(u)h:=u_{k,i}L_{z_{kj}}h_{i,j}-L\delta_{ij}h_{i,j}=(u_{k,i}L_{z_{kj}}-L\delta_{ij})h_{i,j}.\]
This motivates the following definition of the energy-momentum tensor.
\begin{defn}
\label{def:EMT} Let $J$ be a functional satisfying \textsf{(VF)}.
The \emph{energy-momentum tensor} of the variational problem specified
by $J$, is defined by\begin{equation}
T_{ij}(x,y,z)=z_{ki}L_{z_{kj}}(x,y,z)-\delta_{ij}L(x,y,z)\label{eq:12a}\end{equation}
where $x=(x_{i})_{i=1,\cdots,N}\in\Omega$, $y=(y_{k})_{k=1,\cdots,M}\in\mathbb{R}^{M}$,
$z=(z_{ki})_{k=1,\cdots,M;i=1,\cdots,N}\in\mathbb{R}^{NM}$.\end{defn}
\begin{rem}
Notice that the above definition holds for general variables $(x,y,z)\in\Omega\times\mathbb{R}^{M}\times\mathbb{R}^{NM}$,
not just for $(x,u(x),Du(x))$, $x\in\Omega$. Given any vector field
$u\in\mathrm{C}^{1}(\Omega)^{M}$ we have the tensor field\begin{equation}
T_{ij}(x)=T_{ij}(x,u(x),Du(x))=u_{k,i}L_{z_{kj}}(x,u(x),Du(x))-\delta_{ij}L(x,u(x),Du(x)),\label{eq:12w}\end{equation}
for which we will be using the same symbol. Again, in this formula
it is not necessary that $u$ be a solution of the Euler-Lagrange
equations.
\end{rem}

\subsection{\label{sub:2-5}Noether's equations }

Let $u\in\mathrm{C}^{2}(\Omega)$, which is not necessarily a solution
of the Euler-Lagrange equations and $T_{ij}(x)=T_{ij}(x,u(x),Du(x))$.
By the definition of energy-momentum tensor\begin{alignat*}{1}
T_{ij,j} & =\tfrac{\partial}{\partial x_{j}}\left(u_{k,i}L_{z_{kj}}-\delta_{ij}L\right)\\
 & =u_{k,ij}L_{z_{kj}}+u_{k,i}\tfrac{\partial}{\partial x_{j}}L_{z_{kj}}-L_{x_{i}}-L_{y_{k}}u_{k,i}-L_{z_{kj}}u_{k,ij}\\
 & =\left(\tfrac{\partial}{\partial x_{j}}L_{z_{kj}}-L_{y_{k}}\right)u_{k,i}-L_{x_{i}}\end{alignat*}
where $L=L(x,u(x),Du(x))$, $L_{y_{k}}=L_{y_{k}}(x,u(x),Du(x))$ and
$L_{z_{kj}}=L_{z_{kj}}(x,u(x),Du(x))$. From this we obtain\begin{equation}
T_{ij,j}+L_{x_{i}}=\left(\tfrac{\partial}{\partial x_{j}}L_{z_{kj}}-L_{y_{k}}\right)u_{k,i}\label{eq:35}\end{equation}
which motivates the following definition.
\begin{defn}
The system of second order partial differential equations\[
T_{ij,j}(x,u(x),Du(x))+L_{x_{i}}(x,u(x),Du(x))=0\]
or in index-free notation\begin{equation}
\mathrm{div}T(x,u(x),Du(x))+L_{x}(x,u(x),Du(x))=0\label{eq:35a}\end{equation}
is called \emph{Noether's equations}.
\end{defn}
If we define \emph{inner critical points} of $J$ by $\mathfrak{d}J(u)=0$,
i.e. $\mathfrak{d}J(u)h=0$ $\forall h\in\mathcal{D}(\Omega)^{N}$,
then Noether's equations are related to inner critical points in an
analogous manner as critical points to the Euler-Lagrange equations.
Furthermore, by (\ref{eq:35}) every solution $u\in\mathrm{C}^{2}(\Omega)$
of the Euler-Lagrange equations is a solution of Noether's equations.
The converse of this statement is in general not true. Giaquinta and
Hildebrandt \cite{GIA-1} presented the following simple counterexample
to demonstrate this.

\begin{cexample}\label{cexa:2}Let $F\in\mathrm{C}^{1}(\mathbb{R})$,
$F\neq\,$const. and\[
J(u):=\int_{\Omega}F(u(x))dx,\quad u\in\mathrm{C}^{2}(\Omega)\cap\mathrm{C}^{1}(\overline{\Omega}).\]
The energy-momentum tensor is calculated by (\ref{eq:12a})\[
T=-F(u)I,\]
and Noether's equations (\ref{eq:35a}) reduce to\[
F^{\prime}(u)Du=0.\]
It is obvious that every constant function $u=c_{0}$ is a solution
of this system, but not of the Euler-Lagrange equations, which for
this functional assume the form\[
F^{\prime}(u)=0.\]

\end{cexample}

We will show in the next two sections that only trivial counterexamples
are possible for a large class of Lagrangians.

\section{Nonlinear Poisson Equation}

In this section we present the main theorem for the nonlinear Poisson
equation, see equation (\ref{eq:EL-NL-Poi}) below, in a bounded domain
$\Omega$, which can be viewed as the Euler-Lagrange equation of a
variational functional $J$ with the Lagrangian\begin{equation}
L(u,z)=\tfrac{1}{2}|z|^{2}+F(u),\label{eq:Lagr-NL-Poi}\end{equation}
where $z$ corresponds to $Du$ when $u$ is a $\mathrm{C}^{1}$ function
and $F:\mathbb{R}\to\mathbb{R}$. When $F\in\mathrm{C}^{1}(\mathbb{R})$,
$J$ clearly conforms to requirements \textsf{(VF)}.
\begin{thm}
\label{thm:NL-Poi-1}Let $\Omega$ be a bounded domain of $\mathbb{R}^{N}$,
$F\in\mathrm{C}^{1}(\mathbb{R})$, $L$ a Lagrangian of the form (\ref{eq:Lagr-NL-Poi})
and $u\in\mathrm{C}^{2}(\Omega)\cap\mathrm{C}^{1}(\overline{\Omega})$
a non-trivial classical solution of Noether's equations\begin{equation}
\mathrm{div}\, T(u,Du)=0.\label{eq:Noe-NL-Poi}\end{equation}
where $T$ is the energy-momentum tensor corresponding to the Lagrangian
$L$ with components $T_{ij}=u_{,i}u_{,j}-\delta_{ij}L\,$. Then $u$
is a solution of the Euler-Lagrange equation\begin{equation}
\Delta u=F^{\prime}(u)\label{eq:EL-NL-Poi}\end{equation}
in $\Omega$.\end{thm}
\begin{rem}
\label{rem:2}By (\ref{eq:35}), equation (\ref{eq:Noe-NL-Poi}) is
equivalently written in the form\begin{equation}
(\Delta u-F^{\prime}(u))Du=0.\label{eq:Noe-NL-Poi-Alt}\end{equation}

The non-triviality condition $Du\neq0$ means $Du$ is not identically
0, i.e. there is a $x_{0}\in\Omega$ such that $Du(x_{0})\neq0$.\end{rem}
\begin{proof}
Let $u$ be a solution of (\ref{eq:Noe-NL-Poi}) and set\[
\mathrm{A}_{0}:=\{x\in\Omega:\; Du(x)=0\}\]
and\[
\mathrm{A}_{1}:=\{x\in\Omega:\; Du(x)\not=0\}.\]
Obviously $\mathrm{A}_{0}$ is closed relatively $\Omega$, $\mathrm{A}_{1}$
is open and $\mathrm{A}_{0}\cup\mathrm{A}_{1}=\Omega$. It is clear
that (\ref{eq:EL-NL-Poi}) is satisfied in $\mathrm{A}_{1}$. We have
to show (\ref{eq:EL-NL-Poi}) is also satisfied in $\mathrm{A}_{0}$.\smallskip{}

\textsf{Step 1}. Let $\mathrm{D}$ be the subset of $\Omega$ in which
the Euler-Lagrange equation (\ref{eq:EL-NL-Poi}) is satisfied, i.e.\[
\mathrm{D}:=\{x\in\Omega:\;\Delta u(x)=f(u(x))\},\]
where $f:=F^{\prime}$. $\mathrm{D}$ is obviously closed relatively
$\Omega$ and we have already shown that\[
\mathrm{A}_{1}\subset\mathrm{D}.\]
From this, keeping in mind that closures and frontiers are taken relatively
$\Omega$, it follows immediately that $\overline{\mathrm{A}}_{1}\subset\mathrm{D}$,
hence also\[
\partial\mathrm{A}_{1}\subset\mathrm{D}.\]
It is our intention to show that\begin{equation}
\partial\mathrm{A}_{0}\subset\mathrm{D}.\label{eq:1}\end{equation}
For this purpose we will show $\partial\mathrm{A}_{0}=\partial\mathrm{A}_{1}$,
from which (\ref{eq:1}) follows immediately. Indeed, from $\partial\mathrm{A}_{0}=\mathrm{A}_{0}\backslash\overset{\circ}{\mathrm{A}}_{0}$
and \[
\partial\mathrm{A}_{1}=\overline{\mathrm{A}}_{1}\backslash\overset{\circ}{\mathrm{A}}_{1}=(\Omega\backslash\overset{\circ}{\mathrm{A}}_{0})\backslash(\Omega\backslash\mathrm{A}_{0})=(\Omega\backslash\overset{\circ}{\mathrm{A}}_{0})\cap\mathrm{A}_{0}=\mathrm{A}_{0}\backslash\overset{\circ}{\mathrm{A}}_{0}\]
we get $\partial\mathrm{A}_{0}=\partial\mathrm{A}_{1}$ and with this
the validity of (\ref{eq:1}).

If $\overset{\circ}{\mathrm{A}}_{0}=\emptyset$ we are finished, for
$\mathrm{A}_{0}=\partial\mathrm{A}_{0}\subset\mathrm{D}$. Let $\overset{\circ}{\mathrm{A}}_{0}\not=\emptyset$.
By hypothesis, for all $x\in\overset{\circ}{\mathrm{A}}_{0}$ we have
$Du(x)=0$, hence also $D^{2}u(x)=0$ on $\overset{\circ}{\mathrm{A}}_{0}$
and $u(x)=$ const. on connected components of $\mathrm{A}_{0}$.\smallskip{}

\textsf{Step 2}. Fix $x_{0}\in\overset{\circ}{\mathrm{A}}_{0}$. We
will show that there is a $x_{1}\in\partial\mathrm{A}_{0}$ and a
continuous curve $\gamma:\overline{\mathrm{I}}\to\mathrm{A}_{0}$,
$\mathrm{I}=]0,1[$, such that $\gamma(0)=x_{0}$, $\gamma(1)=x_{1}$
and $\gamma([0,1[)\subset\overset{\circ}{\mathrm{A}}_{0}$, i.e. the
curve lies in the interior of $\mathrm{A}_{0}$, with the exception
of $x_{1}$. Let $y\in\mathrm{A}_{1}\neq\emptyset$ by hypothesis
and $\alpha:\overline{\mathrm{I}}\to\Omega$ a continuous curve connecting
$x_{0}=\alpha(0)$ and $y=\alpha(1)$. The set\[
\Gamma:=\{\alpha(t):\, t\in\overline{\mathrm{I}},\;\alpha(t)\in\partial\mathrm{A}_{0}\}\]
is not empty (\cite{DIEU-1}, (3.19.9) and following Remark, p. 70).
Since $\{t\in\overline{\mathrm{I}}:\;\alpha(t)\in\partial\mathrm{A}_{0}\}=\alpha^{-1}(\partial\mathrm{A}_{0})$
is closed, $\tau:=\inf\{t\in\overline{\mathrm{I}}\,:\alpha(t)\in\partial\mathrm{A}_{0}\}$$\in\alpha^{-1}(\partial\mathrm{A}_{0})$,
hence $x_{1}:=\alpha(\tau)\in\partial\mathrm{A}_{0}$ and it is clear
that $\alpha([0,\tau[)\subset\overset{\circ}{\mathrm{A}}_{0}$. For
if there were a $\tau_{1}<\tau$ such that $y^{\prime}=\alpha(\tau_{1})\not\in\overset{\circ}{\mathrm{A}}_{0}$,
then $y^{\prime}\not\in\mathrm{A}_{0}$ and application of the same
procedure for $x_{0}$, $y^{\prime}\in\mathrm{A}_{1}$ would yield
the existence of a $x_{1}^{\prime}=\alpha(\tau^{\prime})\in\partial\mathrm{A}_{0}$
with $\tau^{\prime}<\tau$, which contradicts the definition of $\tau$.
Reparametrisation of $\alpha|[0,\tau]$ yields $\gamma$.\smallskip{}

\textsf{Step 3}. Now let $u(x_{0})=:c_{0}$. Since $x_{0}$ and $x_{1}$
belong to the same connected component of $\mathrm{A}_{0}$, we have
$u(x_{1})=c_{0}$ and $f(u(x_{0}))=f(u(x_{1}))=f(c_{0})=:d_{0}$.
Since by (\ref{eq:1}) $x_{1}\in\mathrm{D}$, we have\begin{equation}
\Delta u(x_{1})=f(u(x_{1}))=d_{0}.\label{eq:2}\end{equation}
But \begin{equation}
\Delta u(x_{1})=\Delta u(\lim_{t\to1-}\gamma(t))=\lim_{t\to1-}\Delta u(\gamma(t))=0\label{eq:3}\end{equation}
for $\gamma(t)\in\overset{\circ}{\mathrm{A}}_{0}$ for all $t\in[0,1[$.
Combination of (\ref{eq:2}) and (\ref{eq:3}) yields $d_{0}=0$,
hence\begin{equation}
f(u(x_{0}))=0.\label{eq:3-1}\end{equation}
This means in particular \[
\Delta u(x_{0})-f(u(x_{0}))=0.\]
With this we have proved $\overset{\circ}{\mathrm{A}}_{0}\subset\mathrm{D}$
and by (\ref{eq:1}) $\mathrm{A}_{0}\subset\mathrm{D}$.
\end{proof}
From the proof of this theorem we conclude without difficulty the
following Corollary.
\begin{cor}
\label{cor:NL-Poi-1}Under the assumptions of Theorem \ref{thm:NL-Poi-1},
if the set $\mathrm{A}_{0}:=\{x\in\Omega:\; Du(x)=0\}$ has an interior
point, then $f(u)=0$ on all of $\overset{\circ}{\mathrm{A}}_{0}$.\end{cor}
\begin{proof}
By hypothesis $\mathrm{A}_{1}\not=\emptyset$. Let $x_{0}\in\overset{\circ}{\mathrm{A}}_{0}$.
Now \textsf{Steps 2} and \textsf{3} of the proof of Theorem \ref{thm:NL-Poi-1}
apply and from (\ref{eq:3-1}) it follows that $f(u(x_{0}))=0$. Since
$x_{0}$ was arbitrary, the assertion is proved.
\end{proof}
As an application, we state the following result for the nonlinear
Poisson equation.
\begin{cor}
\label{cor:NL-Poi-2}Let $\Omega$ be a bounded domain of $\mathbb{R}^{N}$
and $f\in\mathrm{C}(\mathbb{R})$ such that $f(t)\neq0$ for all $t\in\mathbb{R}$.
Then for every solution $u\in\mathrm{C}^{2}(\Omega)\cap\mathrm{C}^{1}(\overline{\Omega})$
of the nonlinear Poisson equation\begin{equation}
\Delta u=f(u)\label{eq:3-2}\end{equation}
 the set $\mathrm{A}_{0}:=\{x\in\Omega:\; Du(x)=0\}$ has no interior
point.\end{cor}
\begin{proof}
Equation (\ref{eq:3-2}) is the Euler-Lagrange equation of the functional
(\ref{eq:Lagr-NL-Poi}) with $F(t)=\int_{0}^{t}f(s)ds$, for which
we have $F\in\mathrm{C}^{1}(\mathbb{R})$. If $\overset{\circ}{\mathrm{A}}_{0}\neq\emptyset$,
by Corollary \ref{cor:NL-Poi-1} we would have $f(u)=0$ on $\overset{\circ}{\mathrm{A}}_{0}$,
which is absurd.\end{proof}
\begin{rem}
In a too strict sense, Noether's equations would never be equivalent
to the Euler-Lagrange equations, if $F^{\prime}(c)\not=0$ for some
$c\in\mathbb{R}$. For by Remark \ref{rem:2} constant functions are
always solutions of Noether's equations, and then the assumption of
equivalence would lead to $F^{\prime}(c)=0$, by considering the constant
function $u=c$, which as said is a solution of Noether's equations.
Aside from the fact that constant solutions are of little practical
importance, since, taking the example of quantum theory, they may
represent only special instances of physical states, or they might
be not integrable when $\Omega$ is not bounded; they never occur
when one considers non-constant boundary conditions. Equivalence of
Euler-Lagrange and Noether's equations is achieved, if one directly
excludes constant solutions or imposes additional hypotheses such
as boundary conditions as in the following Theorem.\end{rem}
\begin{thm}
\label{thm:2}Let $\Omega$ be a bounded domain of $\mathbb{R}^{N}$,
$F\in\mathrm{C}^{1}(\mathbb{R})$, $g\in\mathrm{C}(\partial\Omega)$
a non-constant function and $\mathcal{H}:=\{u\in\mathrm{C}^{2}(\Omega)\cap\mathrm{C}^{1}(\overline{\Omega}):\, u|\partial\Omega=g\}$.
Then the Euler-Lagrange and Noether equations for the Lagrangian (\ref{eq:Lagr-NL-Poi})
are equivalent in $\mathcal{H}$.\end{thm}
\begin{proof}
It follows immediately by Theorem \ref{thm:NL-Poi-1}, for every $u\in\mathcal{H}$
is non-trivial.\end{proof}
\begin{rem}
\label{rem:4}We have restricted the above discussion to bounded domains
only for the sake of convenience. Indeed, this hypothesis serves to
maintain integrability in \textsf{(VF)} and guarantee the exchange
of differentiation and integration. The same purpose also serves the
hypothesis $u\in\mathrm{C}^{1}(\overline{\Omega})$ with the exception
of Theorem \ref{thm:2}. Thus the above results, with proper modifications,
are applicable to unbounded domains as well.
\end{rem}

\section{More General Lagrangians}

We proceed to generalizing the results of the previous section by
considering Lagrangians of the form $L(x,u,z)$, which, along with
\textsf{(VF)}, satisfy the condition \textsf{(H)} below. Again, $u$
is a scalar function and argument $z$ corresponds to $\nabla u$.\medskip{}

\begin{tabular}{cl}
\multicolumn{1}{c}{\textsf{(H)}} & \quad{} $L_{x_{i}z_{i}}(x,u,0)=0$ for all $x,u$.\tabularnewline
\multicolumn{1}{c}{} & \quad{} $L_{x_{i}u}(x,u,0)=0$ for all $x,u$ and all $i=1,\cdots,N$.\tabularnewline
\end{tabular}\medskip{}

\begin{example}
\textsf{(i)} All Lagrangians which are independent of $x$ satisfy
\textsf{(H)}. In particular the Lagrangians of classes I and II in
\cite{ALFAL} satisfy \textsf{(H)}.
\end{example}
\textsf{(ii)} Lagrangians of the form\[
L(x,u,z)=\tfrac{1}{2}\varphi(x,u)|z|^{2}+F(u),\]
where $\varphi:\Omega\times\mathbb{R}\to\mathbb{R}$, satisfy \textsf{(H)}.\hfill{}$\square$

Recall the definition of the Euler-Lagrange derivative, formula (\ref{eq:0}).
\begin{thm}
\label{thm:Gen-1}Let $\Omega$ be a bounded domain of $\mathbb{R}^{N}$,
$L\in\mathrm{C}^{2}(\Omega\times\mathbb{R}\times\mathbb{R}^{N})$
a Lagrangian satisfying (H) and $u\in\mathrm{C}^{2}(\Omega)\cap\mathrm{C}^{1}(\overline{\Omega})$
a non-trivial classical solution of Noether's equations\begin{equation}
\mathrm{div}\, T(x,u,Du)+L_{x}(x,u,Du)=0.\label{eq:Noe-Gen}\end{equation}
Then $u$ is a solution of the Euler-Lagrange equation\begin{equation}
\tfrac{\partial}{\partial x_{j}}L_{u_{,j}}-L_{u}=0\label{eq:EL-Gen}\end{equation}
in $\Omega$.\end{thm}
\begin{rem}
By (\ref{eq:35}), equation (\ref{eq:Noe-Gen}) is equivalently written
in the form\[
\left(\tfrac{\partial}{\partial x_{j}}L_{u_{,j}}-L_{u}\right)u_{,i}=0\]
or in index-free notation\begin{equation}
\delta L(u)\cdot Du=0.\label{eq:Noe-Gen-Alt}\end{equation}
Note that (\ref{eq:Noe-Gen}) is a second order system of partial
differential equations in $u$ and (\ref{eq:EL-Gen}) is a single
second order partial differential equation in $u$.\end{rem}
\begin{proof}
The proof begins exactly as the proof of Theorem \ref{thm:NL-Poi-1}
up to \textsf{Step 3} where the proof of existence of the curve $\gamma$
is complete, with the obvious modification\[
\mathrm{D}:=\{x\in\Omega:\;\delta L(u)(x)=0\}.\]

For fixed $x_{0}\in\overset{\circ}{\mathrm{A}}_{0}$ let $u(x_{0})=:c_{0}$.
Further let $f:=L_{u}$. Since $x_{0}$ and $x_{1}$ belong to the
same connected component of $\mathrm{A}_{0}$, we have $u(x_{1})=c_{0}$
and by \textsf{(H)}\begin{equation}
f(x_{0},u(x_{0}),0)=f(x_{1},u(x_{0}),0)=f(x_{0},c_{0},0)=:d_{0}.\label{eq:4-2}\end{equation}
We have\begin{alignat}{1}
\left.\tfrac{\partial}{\partial x_{i}}L_{z_{i}}(x,u,Du)\right|_{x_{0}} & =L_{x_{i}z_{i}}(x_{0},c_{0},0)+L_{uz_{i}}(x_{0},c_{0},0)u_{,i}(x_{0})+\nonumber \\
 & \hspace{1.4em}L_{z_{i}z_{j}}(x_{0},c_{0},0)u_{,ij}(x_{0})\nonumber \\
 & =L_{x_{i}z_{i}}(x_{0},c_{0},0)=0\label{eq:4-1}\end{alignat}
by \textsf{(H)}. Since by (\ref{eq:1}) $x_{1}\in\mathrm{D}$, we
have in a similar fashion\begin{alignat}{1}
\left.\tfrac{\partial}{\partial x_{i}}L_{z_{i}}(x,u,Du)\right|_{x_{1}} & =L_{x_{i}z_{i}}(x_{1},c_{0},0)+L_{uz_{i}}(x_{1},c_{0},0)u_{,i}(x_{1})+\nonumber \\
 & \hspace{1.4em}L_{z_{i}z_{j}}(x_{1},c_{0},0)u_{,ij}(x_{1})\nonumber \\
 & =L_{x_{i}z_{i}}(x_{1},c_{0},0)=f(x_{1},c_{0},0)=d_{0}\label{eq:4}\end{alignat}
where the second equality from the end follows from the Euler-Lagrange
equations:\[
\left.\tfrac{\partial}{\partial x_{i}}L_{z_{i}}(x,u,Du)\right|_{x_{1}}=L_{u}(x_{1},u(x_{1}),Du(x_{1}))=f(x_{1},c_{0},0).\]
Note that we have omitted a step involving the limiting process $t\to1-$
along $\gamma$ and the associated continuity argument analogous to
that applied in the proof of the following equation. From the first
of \textsf{(H)} and $\gamma(t)\in\overset{\circ}{\mathrm{A}}_{0}$
for all $t\in[0,1[$ we obtain \begin{alignat*}{1}
L_{x_{i}z_{i}}(x_{1},c_{0},0) & =L_{x_{i}z_{i}}(\lim_{t\to1-}\gamma(t),u(\lim_{t\to1-}\gamma(t)),Du(\lim_{t\to1-}\gamma(t)))\\
 & =\lim_{t\to1-}L_{x_{i}z_{i}}(\gamma(t),u(\gamma(t)),Du(\gamma(t)))\\
 & =\lim_{t\to1-}L_{x_{i}z_{i}}(\gamma(t),c_{0},0)=0.\end{alignat*}
Combination of this equation with (\ref{eq:4}) yields $d_{0}=0$,
hence by (\ref{eq:4-2})\[
f(x_{0},c_{0},0)=0.\]
This means in particular \[
\left.\tfrac{\partial}{\partial x_{i}}L_{z_{i}}(x,u,Du)\right|_{x_{0}}-f(x_{0},u(x_{0}),Du(x_{0}))=L_{x_{i}z_{i}}(x_{0},c_{0},0)-f(x_{0},c_{0},0)=0.\]
With this we have proved $\overset{\circ}{\mathrm{A}}_{0}\subset\mathrm{D}$
and by (\ref{eq:1}) $\mathrm{A}_{0}\subset\mathrm{D}$.
\end{proof}
Corollaries \ref{cor:NL-Poi-1} and \ref{cor:NL-Poi-2} transfer to
the general Lagrangians conforming to \textsf{(H),} with the obvious
modifications:
\begin{cor}
Under the assumptions of Theorem \ref{thm:NL-Poi-1}, if the set $\mathrm{A}_{0}:=\{x\in\Omega:\; Du(x)=0\}$
has an interior point, then $L_{u}=0$ on all of $\overset{\circ}{\mathrm{A}}_{0}$.
\end{cor}

\begin{cor}
Let $\Omega$ be a bounded domain of $\mathbb{R}^{N}$ and $L_{u}(x,u,z)\neq0$
in $\Omega\times\mathbb{R}\times\mathbb{R}^{N}$. Then for every solution
$u\in\mathrm{C}^{2}(\Omega)\cap\mathrm{C}^{1}(\overline{\Omega})$
of the partial differential equation in $u$\[
\tfrac{\partial}{\partial x_{j}}L_{u_{,j}}-L_{u}=0\]
 the set $\mathrm{A}_{0}:=\{x\in\Omega:\; Du(x)=0\}$ has no interior
point.
\end{cor}
Remark \ref{rem:4} holds as it is. Finally, we can state the equivalence
of Euler-Lagrange and Noether equations as a general theorem by imposing
conditions analogous to Theorem \ref{thm:2}.
\begin{thm}
\label{thm:4}Let $\Omega$ be a bounded domain of $\mathbb{R}^{N}$,
$g\in\mathrm{C}(\partial\Omega)$ a non-constant function and $\mathcal{H}:=\{u\in\mathrm{C}^{2}(\Omega)\cap\mathrm{C}^{1}(\overline{\Omega}):\, u|\partial\Omega=g\}$.
Then the Euler-Lagrange and Noether equations for a Lagrangian satisfying
(H) are equivalent in $\mathcal{H}$.\end{thm}
\begin{example}
We consider the Lagrangian of p-Laplacian type \cite{CGS}\[
L(u,z)=\frac{1}{2}\varphi(|z|^{2})+F(u)\]
where $F\in\mathrm{C}(\mathbb{R})$ and $\varphi\in\mathrm{C}^{2}(\mathbb{R}^{+})$
such that $\varphi(0)=0$ and $\varphi^{\prime}(s)\geqslant0\;\forall s\geqslant0$,
which satisfies condition \textsf{(H)}. The energy-momentum tensor
for this Lagrangian is given by\[
T_{ij}=\varphi^{\prime}(|Du|^{2})u_{,i}u_{,j}-\delta_{ij}\left(\dfrac{1}{2}\varphi(|Du|^{2})+F(u)\right).\]
By Theorem \ref{thm:4} the equations\[
\mathrm{div}\,(\varphi^{\prime}(|Du|^{2})Du\otimes Du)-D\left(\dfrac{1}{2}\varphi(|Du|^{2})+F(u)\right)=0\]
and\[
\tfrac{\partial}{\partial x_{i}}\left(\varphi^{\prime}(|Du|^{2})\tfrac{\partial u}{\partial x_{i}}\right)=F^{\prime}(u)\]
with the boundary condition $u|\partial\Omega=g$, where $g\in\mathrm{C}(\partial\Omega)$
is a non-constant function, are equivalent.
\end{example}

\section{\label{sec:Apps}Applications}

\subsection{Equivalence of admissible and inner variations}

The $\mathrm{C}^{2}$ solutions of Euler-Lagrange (respectively Noether)
equations are critical (respectively inner critical) points of the
corresponding variational functional. Since each inner variation gives
rise to an admissible variation \cite{FAL}\[
w=Du\cdot h,\]
the question arises if these two types of variation are equivalent.
In this case, one could work with the smaller set of inner variations
instead of the larger set of admissible variations. This might have
an impact on the numerical computation of solutions. Next theorem
gives an answer to this question.
\begin{thm}
\label{thm:5}Let $J$ be a variational functional with Lagrangian
$L\in\mathrm{C}^{2}(\Omega\times\mathbb{R}\times\mathbb{R}^{N})$
satisfying (H) and $u\in\mathrm{C}^{2}(\Omega)\cap\mathrm{C}^{1}(\overline{\Omega})$
a nontrivial function. Then the variational problem\begin{equation}
\delta J(u)v=0\quad\forall v\in\mathcal{D}(\Omega)\label{eq:6}\end{equation}
is equivalent to the problem\begin{equation}
\delta J(u)w=0\quad\forall w\in\mathcal{I}(\Omega)\label{eq:7}\end{equation}
where $\mathcal{I}(\Omega):=\{Du\cdot h:\, h$ \emph{is an inner variation
of} $u\}$. In simple words one can consider only inner variations
of $u$ in Problem (\ref{eq:6}).\end{thm}
\begin{proof}
Equation (\ref{eq:7}) follows immediately from (\ref{eq:6}) when
$u$ is $\mathrm{C}^{\infty}$. Otherwise $w\in\mathrm{C}_{c}^{1}(\Omega)$
while (\ref{eq:6}) requires $v\in\mathcal{D}(\Omega)$. In this case
the proof follows by a simple density argument. For the converse,
if (\ref{eq:7}) is valid, then $\delta J(u)Du\cdot h=0$ for all
$h\in\mathcal{D}(\Omega)^{N}$. By the fundamental lemma of calculus
of variations $\delta J(u)Du=0$, hence $u$ is a solution of Noether's
equations and by Theorem \ref{thm:Gen-1} also a solution of the Euler-Lagrange
equations. Since $u$ is $\mathrm{C}^{2}$, this means $u$ is a critical
point of $J$, i.e. (\ref{eq:6}) is satisfied.\end{proof}
\begin{rem}
For $w=Du\cdot h$, $h\in\mathcal{D}(\Omega)^{N}$, we can prove by
combining formulas (\ref{eq:10}), (\ref{eq:12a}) and (\ref{eq:35})
that equation (\ref{eq:7}) may be written in the form $\mathfrak{d}J(u)h=0$
and recalling formula (\ref{eq:0-1}), $\left.\tfrac{d}{dt}J(u\circ\xi_{h}^{t})\right|_{t=0}=0$.
This admits of a possibly interesting interpretation: if one knows
the distribution function of $u$, in order to solve the second order
partial differential equation $\delta J(u)=0$, one has just to place
the values of $u$ in their right position.
\end{rem}

\subsection{Determination of Lagrangian from conservation law}

As a second application, we present an example of {}``reverse engineering''
the Lagrangian of a physical problem, mentioned in the Introduction.
By this we mean there are experimental data available for a physical
system and one would like to determine the system, in our case the
Lagrangian. This is obviously a kind of inverse problem and, as it
is well-known, such problems may have multiple solutions, or no solution
at all and they may be not well-defined in certain senses, as for
example the solution may not depend continuously on the data.

Assume there is a physical quantity $u$, which, as a conclusion from
the examination of experimental data, is conserved according to the
following physical law: the integral of the vector quantity $\mathbf{q}=\mathrm{div}\,\mathbf{T}$
over any set $V$ of three dimensional space remains constant and
equal to 0; $\mathbf{T}$ being a tensor with components $T_{ij}=u_{,i}u_{,j}-\delta_{ij}\tfrac{1}{2}(|\nabla u|^{2}-\alpha u^{2})$
in an orthogonal coordinate system and $\alpha$ a physical constant.
In reality, this might be the diffusion of a rarefied substance contained
in concentration $u$ in a medium, which is accompanied by degradation
of the substance according to a first order reaction, or the conduction
of heat in the presence of heat production or destruction at a rate
proportional to the temperature $u$. The question is, if from these
data one could determine a Lagrangian so that the Euler-Lagrange equations
obtained from this Lagrangian coincide with the field equations deduced
from experimental data.

Assuming that the system is confined in $\Omega$, from the conservation
law obeyed by the field $u$ we obtain $\mathrm{div}\,\mathbf{T}(x)=0$
for any $x\in\Omega$. This is the experimentally determined field
equations for $u$. If we assume that there is a Lagrangian $L=L(u,z)$
such that the energy-momentum tensor is $\mathbf{T}$, Theorem \ref{thm:NL-Poi-1}
guarantees that the corresponding Euler-Lagrange equations will be
satisfied for every non-constant scalar field $u$ satisfying $\mathrm{div}\,\mathbf{T}(x)=0$.
Thus, it remains to be checked if there is a Lagrangian having as
energy-momentum tensor the tensor $\mathbf{T}$ given above. By formula
(\ref{eq:12a}) $L$ must satisfy the following first order system
of partial differential equations\begin{equation}
z_{i}\frac{\partial L}{\partial z_{j}}-\delta_{ij}L=z_{i}z_{j}-\delta_{ij}\tfrac{1}{2}(|z|^{2}-\alpha u^{2})\label{eq:8}\end{equation}
for any $i,j\in\{1,2,3\}$. We note in passing that this is an overdetermined
system, which may have no solutions. To solve (\ref{eq:8}) we observe
that from\[
z_{i}\frac{\partial L}{\partial z_{j}}=z_{i}z_{j},\quad i\neq j\]
we get\begin{equation}
L(u,z)=\tfrac{1}{2}|z|^{2}+F(u)\label{eq:8-1}\end{equation}
where $F$ is a function to be determined from the additional equations\begin{equation}
z_{i}\frac{\partial L}{\partial z_{i}}-L=z_{i}^{2}-\tfrac{1}{2}(|z|^{2}-\alpha u^{2})\label{eq:9}\end{equation}

for $i\in\{1,2,3\}$; Einstein's summation convention does not apply
in (\ref{eq:9}). From (\ref{eq:9}) on substituting $L$ by (\ref{eq:8-1})
we obtain $F(u)=-\tfrac{\alpha}{2}u^{2}$ and then\begin{equation}
L(u,z)=\tfrac{1}{2}|z|^{2}-\tfrac{\alpha}{2}u^{2},\label{eq:9-1}\end{equation}

which is actually a solution of (\ref{eq:8}), as it may be directly
verified. It follows from the way $L$ was determined, that (\ref{eq:9-1})
is a unique solution of (\ref{eq:8}).

\end{document}